\documentclass[11pt]{amsart}
\usepackage[margin=1in]{geometry}  
\usepackage{graphicx,upref}

\usepackage{amsthm}

\usepackage{mathtools, tikz, tikz-cd}
\usepackage{amsmath, amssymb, amsfonts, amsxtra}
\usepackage[colorlinks,
             linkcolor=black!75!red,
             citecolor=blue,
             pdftitle={},
             pdfauthor={},
             pdfproducer={pdfLaTeX},
             bookmarksopen=true,
             bookmarksnumbered=true]{hyperref}

\newtheorem{theorem}{Theorem}[section]
\newtheorem{proposition}[theorem]{Proposition}
\newtheorem{lemma}[theorem]{Lemma}
\newtheorem{corollary}[theorem]{Corollary}

\theoremstyle{definition}
\newtheorem{definition}[theorem]{Definition}
\newtheorem{example}[theorem]{Example}
\newtheorem{remark}[theorem]{Remark}

\title{Admissibility of C*-Covers for Operator Algebra Dynamical Systems} 
\author{Mitch Hamidi}  
\address{Department of Mathematics, Embry-Riddle Aeronautical University, Prescott, AZ 86301-3720, USA} 
\email{hamidim@erau.edu}

\keywords{$\al$-admissibility, C*-cover, crossed product, operator algebra, non-selfadjoint}

\subjclass[2010]{47L55, 46L55, 47L65}

\usepackage[mathscr]{euscript}

\newcommand{\Z}{\mathbb{Z}}
\newcommand{\C}{\mathbb{C}}

\newcommand{\T}{\mathbb{T}} 
\newcommand{\D}{\mathbb{D}} 

\newcommand{\al}{\alpha}
\newcommand{\ep}{\varepsilon}

\newcommand{\lp}{\left(}
\newcommand{\rp}{\right)}

\newcommand{\set}[1]{\left\{ #1 \right\}}

\newcommand{\norm}[1]{\left\lVert#1\right\rVert}

\newcommand{\isom}{\cong}

\newcommand{\Aut}{\text{Aut}} 
 
\newcommand{\semi}{\rtimes} 
\newcommand{\Hil}{\mathcal{H}} 
\newcommand{\BH}{\mathcal{B}(\Hil)}

\newcommand{\A}{\mathcal{A}}
\newcommand{\B}{\mathcal{B}}

\newcommand{\Cmax}[1]{C^*_\text{max}( #1)} 
\newcommand{\Ce}[1]{C^*_{e}( #1)} 

\newcommand{\Cc}{\mathcal{C}} 
\newcommand{\J}{\mathcal{J}} 
\newcommand{\imax}{i_{\text{max}}}
\newcommand{\imin}{i_{\text{min}}}

\begin{document} 
 
\begin{abstract}  
We characterize when a C*-cover admits a C*-dynamical extension of dynamics on an operator algebra in terms of the boundary ideal structure for the operator algebra in its maximal representation and show that the C*-covers that admit such an extension form a complete lattice. We study dynamical systems arising from groups acting via inner automorphisms in a C*-cover and produce an example of a C*-cover that admits no extension of dynamics on a finite-dimensional non-self-adjoint operator algebra. We construct a partial action on a class of C*-covers that recovers the crossed product of an operator algebra as a subalgebra of the partial crossed product, even when the C*-cover admits no dynamical extension.
\end{abstract} 
\maketitle


 
\section{Introduction}

In \cite{Kat-1}, E. Katsoulis and C. Ramsey introduced the construction of a non-self-adjoint crossed product that encodes the action of a group of automorphisms on an operator algebra. The construction is a tool for generating examples of non-self-adjoint operator algebras and provides a framework to answer open problems in both non-self-adjoint and self-adjoint operator algebra theory. For example:
\begin{itemize}
    \item the \textit{Hao-Ng isomorphism problem} was resolved for (i) reduced crossed products and all discrete groups in \cite{Kat-haong-2017} and (ii) reduced crossed products and all groups on hyperrigid C*-correspondences in \cite[Theorem 3.9]{Kat-3},
    \item a program to resolve the \textit{Hao-Ng isomorphism problem} for full crossed products was established in \cite{Kat-3} -- see Theorem 4.9,
    \item Takai duality has been generalized to arbitrary operator algebra crossed products \cite[Theorem 4.4]{Kat-1},
    \item it was shown in \cite[Corollary 5.15]{Kat-1} that there exist semi-Dirichlet operator algebras that are not Dirichlet and do not arise as the tensor algebra of a C*-correspondence.
\end{itemize}

E. Katsoulis and C. Ramsey's crossed product stands apart from other non-self-adjoint crossed product constructions such as those motivated by W. Arverson in the 1960s \cite{Arveson-semicrossed-1,Arveson-semicrossed-2} and formalized by J. Peters in the 1980s \cite{Peters-84}. These particular non-self-adjoint crossed products, called a \textit{semicrossed products}, encode the action of a semigroup of endomorphisms on an operator algebra. While the impact of \cite{Arveson-semicrossed-1,Arveson-semicrossed-2,Peters-84} and subsequent work by many authors has had broad impact, the current paper focuses on the ``road less traveled'' -- actions of groups on non-self-adjoint operator algebras.

Given an action $\al:G\curvearrowright \A$ of a group $G$ of completely isometric automorphisms on an operator algebra $\A$ and a completely isometric representation $j:\A\to \Cc$ of $\A$ into a C*-algebra $\Cc=C^*(j(\A))$, the crossed product of $\A$ by $G$ is defined to be the norm-closure of $C_c(G,j(\A))$ in a C*-crossed product
of $\Cc$ by $G$, pending $\al$ extends to an appropriate action of $G$ on $\Cc$. When such an extension exists, the \textit{C*-cover} $(\Cc,j)$ is called \textit{$\al$-admissible}. The existence of such dynamical extensions in the \textit{extremal representations} (Remark \ref{extremal-reps}) of an operator algebra is sufficient in E. Katsoulis and C. Ramsey's original work, but little is known about the existence of admissible extensions in the non-extremal setting; it is unclear which, if any, non-extremal C*-covers are $\al$-admissible.

In this paper, we develop a program for determining the $\al$-admissibility of a given C*-cover. Section \ref{sec:preliminaries} contains preliminary information about $C*$-covers, $\al$-admissibility of C*-covers, and non-self-adjoint crossed products. This section draws from \cite{Kat-1}, but is included for the reader's convenience. In Section \ref{sec:structure-of-cstar-covers}, we compile some basic theory and folklore results for C*-covers and prove that the set of equivalence classes of $\al$-admissible C*-covers forms a complete lattice. The main result of Section \ref{sec:structure-of-cstar-covers} is Theorem \ref{thm:admisscorrespondence}, which characterizes $\al$-admissibility of a C*-cover in terms of boundary ideals that are invariant under $\al$:

\begingroup
\renewcommand{\thetheorem}{\ref{thm:admisscorrespondence}}
\begin{theorem}
Let $(\A, G, \al)$ be a dynamical system and let $(\mathcal{D}, i)$ be an $\al$-admissible C*-cover for $\A$. If $(\mathcal{C},j)$ is any C*-cover for $\A$ such that there exists a $*$-homomorphism $\pi$ of $\mathcal{D}$ onto $\Cc$ satisfying $\pi\circ i = j$, i.e. $(\mathcal{C},j) \preceq (\mathcal{D},i)$, then the following are equivalent:
\begin{enumerate}
\item  $(\mathcal{C}, j)$ is $\al$-admissible.
\item $\ker \pi$ is an $\al$-invariant ideal in $\mathcal{D}$. 
\end{enumerate}
\end{theorem}
\endgroup

Section \ref{sec:inner-dynamical} focuses on group actions arising from automorphisms that are inner in some C*-cover. We prove that the existence of such an action yields $\al$-admissibility in Theorem \ref{thm:locally-inner-admiss}. 

Section \ref{sec:ex-non-admiss} marks a departure from the previous sections by producing the first finite-dimensional example of a non-admissible C*-cover. The example highlights that $\al$-admissibility is far from guaranteed for a C*-cover. In particular, we construct a non-extremal finite-dimensional C*-cover that is not admissible for a dynamical system arising from the action of $\Z/2\Z$ by automorphisms inner to the C*-envelope.

We conclude with Section \ref{sec:partial-actions} by suggesting a method to recover dynamics in a potentially non-admissible C*-cover. Our main result is Theorem \ref{thm:partial-action} which gives sufficient conditions on the Shilov boundary ideal for the crossed product of an operator algebra to be recovered via a partial action on a potentially non-admissible C*-cover:

\begingroup
\renewcommand{\thetheorem}{\ref{thm:partial-action}}
\begin{theorem}
Let $G$ be a discrete group and let $(\A, G, \al)$ be a dynamical system. Suppose $(\mathcal{C},j)$ is a  C*-cover for $\mathcal{A}$ such that the Shilov boundary ideal $\J$ for $\mathcal{A}$ in $\mathcal{C}$ is maximal and not essential. If $G$ is amenable or $\A$ is Dirichlet, then there exists a partial action $\theta: G \curvearrowright \mathcal{C}$ and a norm closed subalgebra $\B \subseteq \mathcal{C}\semi_\theta G$ such that $\B$ is completely isometrically isomorphic to $\A\semi_\al G.$
\end{theorem}
\endgroup

\begin{remark}
    While this paper was in preparation, the author has been delighted to see a growth of interest in C*-covers as objects of study in their own right. Several authors have announced exciting work (\cite{thompson2023maximal}, \cite{humeniuk2023lattice}, \cite{humeniuk2023crossed}) that builds upon results from the author's PhD thesis that appear in this paper. 
\end{remark}

\section{Preliminaries}\label{sec:preliminaries}

\subsection{Operator Algebras and C*-Covers}
We assume that the reader is familiar with the basic theory of abstract operator algebras. See \cite{Blecher-1} for a comprehensive treatment. Let $(\A, G, \al)$ be an (operator algebra) \textit{dynamical system}, i.e., $\A$ is an approximately unital operator algebra, $G$ is a locally compact Hausdorff group, and $\al:G \to \Aut(\A)$ is a strongly continuous group representation such that $\al_s$ is a completely isometric automorphism for all $s\in G$. A \textit{C*-cover} for $\A$ is a pair $(\Cc, j)$, where $\Cc$ is a C*-algebra and $j:\A \to \Cc$ is a completely isometric homomorphism satisfying $\Cc=C^*(j(\A))$. 

Given two C*-covers $(\Cc_k, j_k)$, $k=1,2$, for $\A$, we say $(\Cc_1, j_1)$ is \textit{below} $(\Cc_2,j_2)$ or $(\Cc_2,j_2)$ is \textit{above} $(\Cc_1,j_1)$, both denoted $(\Cc_1, j_1) \preceq (\Cc_2, j_2)$, if there exists a (necessarily surjective) $*$-homomorphism $\pi:\Cc_2 \to \Cc_1$ such that $\pi \circ j_2 = j_1$. We say two C*-covers $(\Cc_k, j_k)$, $k=1,2$, for $\A$ are \textit{equivalent}, denoted $(\Cc_1, j_1)\simeq (\Cc_2, j_2)$, if there exists a $*$-isomorphism $\pi:\Cc_2 \to \Cc_1$ satisfying $\pi \circ j_2 = j_1$. It is easy to see that this gives an equivalence relation on C*-covers for $\A$ and that $\preceq$ gives a partial ordering on equivalence classes of C*-covers for $\A.$

It is well-known that there exist maximal and minimal (equivalence classes of) C*-covers for $\A$, denoted $\Cmax\A \equiv (\Cmax\A, \imax)$ and $\Ce\A\equiv (\Ce\A, \imin)$, which are defined by their respective universal properties. See Definition 2.4.3 and Theorem 4.3.1 in \cite{Blecher-1} for details. 

\begin{remark}\label{extremal-reps}
In this paper, we refer to $\Cmax\A$ and $\Ce\A$ as the \textit{extremal representations} for an operator algebra $\A.$ Any C*-cover for $\A$ strictly below $\Cmax\A$ and strictly above $\Ce\A$ will sometimes be referred to as a \textit{non-extremal representation} for $\A.$
\end{remark}

Given a C*-cover $(\Cc, j)$ for $\A$, an ideal $J$ in $\Cc$ is a \textit{boundary ideal} for $\A$ in $\Cc$ if the canonical quotient map $q:\Cc \to \Cc/J$ is completely isometric on $j(\A)$. The existence of the equivalence class $\Ce\A$, called the \textit{C*-envelope for $\A$}, is due to the existence of a maximal (with respect to inclusion) boundary ideal for $\A$ in any C*-cover for $\A$. In \cite{Arv-69}, W. Arveson named such an ideal $\J \subseteq \Cc$ the \textit{Shilov boundary ideal} for $\A$ in $\Cc$ and one can show that $(\Cc/\J, q\circ j)\simeq \Ce\A$. See \cite{Ham-79,Drit-Mccull,Arv-08,Davidson-Kennedy} for details about the existence of $\J$ and $\Ce{\A}$. 

\subsection{$\al$-Admissibility and Crossed Products of Operator Algebras}

We summarize the basic construction of a crossed product introduced in \cite{Kat-1}. Given a dynamical system $(\A, G, \al)$, a C*-cover $(\Cc, j)$ for $\A$ is called \textit{$\al$-admissible} if there exists a (necessarily unique -- \cite[Lemma 2.2]{Kat-3}) action of $G$ on $\Cc$ by $*$-automorphisms, $\tilde{\al}:G \to \Aut(\Cc)$, such that for each $s \in G$, the following diagram commutes.

\begin{center}
\begin{tikzcd}
\mathcal{C} \arrow{r}{\tilde{\al}_s}& \mathcal{C} \\
\A \arrow{u}{j} \arrow{r}[swap]{\al_s} & \A\arrow{u}[swap]{j}\\
\end{tikzcd}
\end{center}

E. Katsoulis and C. Ramsey prove in \cite{Kat-1} that $\Ce\A$ and $\Cmax\A$ are always $\al$-admissible due to their respective universal properties.

\begin{definition}
Let $(A, G, \al)$ be a dynamical system and suppose $(\Cc,j)$ is an $\al$-admissible C*-cover for $\A.$ 
\begin{itemize}
    \item The \textit{full crossed product relative to $\Cc$} is the norm closure of $C_c(G, j(\A))$ in the full C*-crossed product $\Cc\semi_{\tilde{\al}}G$, denoted $\A\semi_{(\Cc,j),\al}G$.
    \item The \textit{reduced crossed product relative to $\Cc$} is the norm closure of $C_c(G, j(\A))$ in the reduced C*-crossed product $\Cc\semi_{\tilde{\al}}^rG$, denoted $\A\semi_{(\Cc,j),\al}^r G$.
\end{itemize}
\end{definition}

A \textit{covariant representation} of a dynamical system $(\A, G,\al)$ is a triple $(\pi, U, \Hil)$ consisting of a Hilbert space $\Hil$, a strongly continuous unitary group representation $U$ of $G$ on $\Hil$, and a non-degenerate, completely contractive representation $\pi:\A \to \BH$ satisfying $$U_s\pi(a)U_s^*=\pi(\al_s(a)) \quad \text{for all $s\in G$, $a \in \A$}.$$

The \textit{full crossed product} $\A\semi_\al G$ is defined to be the full crossed product relative to $\Cmax\A$, i.e., $\A\semi_\al G:=\A\semi_{\Cmax\A,\al}G.$ This definition is justified in \cite{Kat-1} as $\A\semi_{\Cmax\A,\al}G$ is the universal operator algebra with respect to covariant representations for $(\A, G, \al)$. 

A \textit{regular covariant representation} of $(\A, G, \al)$ is a (necessarily covariant) representation $(\bar{\pi}, \lambda_\Hil, L^2(G,\Hil))$ induced from some completely contractive representation $\pi:\A\to\BH$, where $\bar{\pi}:\A\to \B(L^2(G,\Hil))$ and $\lambda_\Hil:G \to \mathcal{U}(L^2(G,\Hil))$ are given by $$\bar{\pi}(a)f(t)=\pi(\al_t^{-1}(a))f(t), \quad \lambda_\Hil(s)f(t) = f(s^{-1}t)\quad \text{for all $t\in G, f\in L^2(G, \Hil)$}$$

E. Katsoulis and C. Ramsey prove in \cite[Corollary 3.16]{Kat-1} that every reduced crossed product is isometrically isomorphic, i.e., the reduced crossed product is independent of $\al$-admissible C*-cover. Thus, the \textit{reduced crossed product}, denoted $\A \semi_\al^r G$, is defined to be this object. It is often convenient to identify $\A\semi_\al^r G \equiv \A\semi_{\Ce\A,\al}^r G$ and one can verify that $\A\semi_\al^r G$ is the universal operator algebra with respect to regular covariant representations of $(\A, G, \al).$

\section{Structure of C*-Covers}\label{sec:structure-of-cstar-covers}

The goal of this section is to examine the structure of $\al$-admissible C*-covers for an operator algebra $\A$. We begin by examining the structure of C*-covers for $\A$.

\subsection{The Complete Lattice of C*-Covers}\label{sec:ccover-lattice}

Let $\A$ be an approximately unital operator algebra. Using notation from \cite{humeniuk2023lattice}, define C*-Lat($\A$) to be the set of equivalence classes of C*-covers for $\A$, which is partially ordered with respect to the partial ordering of equivalence classes for C*-covers introduced in Section \ref{sec:preliminaries}.

Recall $\text{C*-Lat($\A$)}$ contains a maximal and minimal element, $\Cmax\A$ and $\Ce\A$, respectively. The existence of $\Cmax\A$ and $\Ce\A$ yields a 1-1 correspondence between C*-covers for $\A$ (up to C*-cover equivalence) and boundary ideals $J$ for $\A$ in $\Cmax\A$.

\begin{proposition}\label{prop:1-1-corr-for-boundary-ideals}
There exists a 1-1 correspondence between equivalence classes of C*-covers for $\A$ and boundary ideals for $\A$ in $\Cmax\A$.
\begin{proof}
If $J$ is a boundary ideal for $\A$ in $\Cmax\A$ and $q:\Cmax\A \to \Cmax\A/J$ is the natural quotient map, it is clear that $(\Cmax\A/J, q\circ \imax)$ is a C*-cover for $\A$. Likewise, if $(\Cc, j)$ is a C*-cover for $\A$, the universal property for $\Cmax\A$ says there exists a $*$-homomorphism $\pi:\Cmax\A \to \Cc$ such that $\pi \circ \imax= j$, which implies $\ker\pi$ is a boundary ideal since $\Cc \isom \Cmax\A/\ker\pi$ and $j=\pi\circ \imax$ is a complete isometry.

One can check that this correspondence is well-defined on equivalence classes using the universal property of $\Cmax{A}$ and then chasing the appropriate diagram.
\end{proof}
\end{proposition}

The existence of the Shilov boundary ideal $\J$ in $\Cmax{\A}$ yields a lattice structure for boundary ideals for $\A$ in $\Cmax\A$ and one can check that the 1-1 correspondence in Proposition \ref{prop:1-1-corr-for-boundary-ideals} is an anti-lattice isomorphism, which implies $\text{C*-Lat($\A$)}$ has a lattice structure corresponding to the lattice structure of boundary ideals for $\A$ in $\Cmax\A$. This is mentioned by D. Blecher et. al at the end of \cite{BMQ-Morita} and Remark 2.4 of \cite{Blecher-1}. 

To establish notation and ensure this work self-contained, we define the appropriate meet and join operations that make $\text{C*-Lat($\A$)}$ a complete lattice. These operations were first studied in the author's PhD thesis (\cite{Hamidiphd}) and by I. Thompson in \cite{thompson2023maximal}. The theory of the complete lattice of C*-covers under these operations has since been advanced by A. Humeniuk and C. Ramsey in \cite{humeniuk2023lattice}.


We begin by showing the direct sum of completely isometric representations of an operator algebra generates a C*-cover and that C*-cover is an upper bound for each summand of the representation.

\begin{proposition}\label{thm:cjoin} Let $\mathscr{S} =\set{(\mathcal{C}_\lambda, j_\lambda): \lambda \in \Lambda}$ be a collection of C*-covers for $\A$ and define $j:=\bigoplus_\lambda j_\lambda$. Then $(C^*(j(\A)), j)$ is a C*-cover for $\A$ that is the supremum of $\mathscr{S}$.
\end{proposition}

\begin{proof}
It is clear that $j$ is a completely isometric algebra homomorphism, and thus, $(C^*(j(\A)), j)$ is a C*-cover for $\A$. One can check that $(C^*(j(\A)), j)$ is an upper bound for $\mathscr{S}$ using projections onto the appropriate summand. Indeed, suppose $(\mathcal{D},i)$ is another upper bound for $\mathscr{S}$. Then for each $\lambda \in \Lambda$, there exists a $*$-homomorphism $\pi_\lambda$ of $\mathcal{D}$ onto $\mathcal{C}_\lambda$ such that $\pi_\lambda \circ i = j_\lambda$. Define $\pi:\mathcal{D} \to \bigoplus_\lambda \mathcal{C}_\lambda$ by $\pi = \bigoplus_\lambda \pi_\lambda$. Then $\pi$ is a $*$-homomorphism of $\mathcal{D}$ onto $C^*(j(\A))$ such that $$\pi \circ i = \bigoplus_\lambda (\pi_\lambda \circ i) = \bigoplus_\lambda j_\lambda = j.$$ Thus, $(C^*(j(\A)), j)$ is below $(\mathcal{D},i)$.
\end{proof}

Proposition \ref{thm:cjoin} gives us the definition of a concrete join operation on $\text{C*-Lat($\A$)}$ that makes it a complete join semi-lattice. It is left to the reader to verify that the operation is well-defined on equivalence classes of C*-covers.

\begin{definition}\label{def:cjoin}
Let $\mathscr{S} =\set{(\mathcal{C}_\lambda, j_\lambda): \lambda \in \Lambda}$ be a collection of C*-covers for $\A$. Consider the C*-subalgebra $C^*\left(\bigoplus_{\lambda} j_\lambda(\A)\right)$ contained in $\bigoplus_{\lambda \in \Lambda} C_\lambda$. We say the \textit{join} of $\mathscr{S}$ is the equivalence class of $(C^*\left(\bigoplus_{\lambda} j_\lambda(\A)\right), \bigoplus_\lambda j_\lambda)$ in $\text{C*-Lat($\A$)}$, which we denote by $\bigvee_{\lambda\in \Lambda} \Cc_\lambda$.
\end{definition}

We define the meet operation for $\text{C*-Lat($\A$)}$ using the boundary ideal structure of $\Cmax{\A}$.

\begin{proposition}\label{thm:cmeet}
Let $\mathscr{S} =\set{(\mathcal{C}_\lambda, j_\lambda): \lambda \in \Lambda}$ be a collection of C*-covers for $\A$. Then there exists a C*-cover for $\A$ that is an infimum for $\mathscr{S}$.
\begin{proof}
Let $\mathcal{J}$ be the Shilov boundary ideal for $\A$ in $\Cmax{\A}$. For each $\lambda \in \Lambda$, there exists a $*$-epimorphism $q_\lambda:\Cmax{\A}\to \mathcal{C}_\lambda$ such that $q_\lambda\circ \imax=j_\lambda$ and $\ker q_\lambda\subseteq \mathcal{J}$ is a boundary ideal for $\A$ in $\Cmax{\A}$.

Define $J$ to be the norm-closure of the two-sided ideal generated by $\bigcup_\lambda \ker q_\lambda$ in $\Cmax{\A}$, i.e. $J= \overline{\sum_\lambda \ker q_\lambda}^{\norm{\cdot}}$. Then $J$ is the smallest closed two-sided in $\Cmax{\A}$ containing $\bigcup_\lambda \ker q_\lambda$, which implies $J \subseteq \mathcal{J}$. Hence, 
$J$ is a boundary ideal for $\A$ in $\Cmax{\A}$, and it follows that $(\Cmax{\A}/J, q_J\circ \imax)$ is a C*-cover for $\A$, where $q_J$ is the canonical quotient map of $\Cmax{\A}$ onto $\Cmax{\A}/J$. 

Clearly, $(\Cmax{\A}/J, q_J\circ \imax)$ is a lower bound for $\mathscr{S}$. We claim that $(\Cmax{\A}/J, q_J\circ \imax)$ is the infimum of $\mathscr{S}$. Suppose $(\mathcal{D},i)$ is a C*-cover for $\A$ such that $(\mathcal{D},i)\preceq (\mathcal{C}_\lambda, j_\lambda)$ for all $\lambda \in \Lambda$. Then there exists a $*$-epimorphism $\pi$ of $\Cmax{\A}$ onto $\mathcal{D}$ such that $\pi \circ \imax = j$. Hence, $\ker \pi$ is a boundary ideal for $\A$ in $\Cmax{\A}$ such that $\Cmax{\A}/\ker\pi \isom \mathcal{D}$. Note that $\bigcup_{\lambda} \ker q_\lambda$ is contained in $\ker \pi$ since $(\mathcal{D}, i)\preceq (\mathcal{C}_\lambda, j_\lambda)$ for all $\lambda \in \Lambda$. Indeed, for each $\lambda \in \Lambda$, $(\mathcal{D}, i) \preceq (\mathcal{C}_\lambda, j_\lambda)$ implies there exists a $*$-epimorphism $Q_\lambda:\mathcal{C}_\lambda \to \mathcal{D}$ such that the following diagram commutes.

\begin{center}
\begin{tikzcd}
 &  \Cmax{\A} \arrow{d}{q_\lambda} \\
 \A \arrow{dr}[swap]{i} \arrow{ur}{\imax} \arrow{r}[swap]{j_\lambda} & \mathcal{C}_\lambda\arrow[dashed]{d}{Q_\lambda}\\
 &  \mathcal{D}
\end{tikzcd}
\end{center}
Thus, for each $\lambda \in \Lambda$, we have $$(Q_\lambda\circ q_\lambda) \circ \imax = Q_\lambda\circ (q_\lambda \circ \imax) = Q_\lambda \circ j_\lambda = i = \pi \circ \imax,$$ which implies $\pi = Q_\lambda\circ q_\lambda$ since $\Cmax\A=C^*(\imax(\A))$. It follows that $\ker q_\lambda \subseteq \ker \pi$ for all $\lambda \in \Lambda$. 

Since $J$ is the smallest ideal containing $\bigcup_\lambda \ker q_\lambda$, we must have $J \subseteq \ker \pi$. Thus, $\ker \pi/J$ is an ideal in $\Cmax{\A}/J$ such that $(\Cmax{\A}/J)/(\ker\pi/J) \isom \mathcal{D}$, i.e. $\mathcal{D}$ is a quotient of $\Cmax{\A}/J$. Hence, there exists a $*$-homomorphism $\tilde{\pi}$ of $\Cmax{\A}/J$ onto $\mathcal{D}$ such that $\tilde{\pi}\circ q_J = \pi$, which yields $$\tilde{\pi}\circ(q_J\circ \imax) = (\tilde{\pi}\circ q_J)\circ \imax = \pi\circ \imax = i.$$ Therefore, we have $(\mathcal{D}, i)\preceq (\Cmax{\A}/J, q_J\circ \imax)$.
\end{proof}
\end{proposition}

Proposition \ref{thm:cmeet} yields the definition of a meet operation on $\text{C*-Lat($\A$)}$ that makes it a complete meet semi-lattice.  It is left to the reader to verify that the operation is well-defined on equivalence classes of C*-covers.

\begin{definition}\label{def:cmeet}
Let $\mathscr{S} =\set{(\mathcal{C}_\lambda, j_\lambda): \lambda \in \Lambda}$ be a collection of C*-covers for $\A$. For each $\lambda\in \Lambda$, let $q_\lambda$ be the quotient map of $\Cmax\A$ onto $\Cc_\lambda$ such that $q_\lambda\circ\imax = j_\lambda$. Define $J$ to be the norm closure of the two-sided ideal ${\sum_\lambda \ker q_\lambda}$ in $\Cmax{\A}$ and let $q_J$ be the quotient map of $\Cmax\A$ onto $\Cmax\A/J$. The \textit{meet} of $\mathscr{S}$ is the equivalence class for the C*-cover $(\Cmax{\A}/J, q_J\circ\imax)$, which we denote by $\bigwedge_{\lambda\in \Lambda} \Cc_\lambda$.
\end{definition}

One can define the join of C*-covers in terms of boundary ideals in the maximal C*-cover as we did with the meet by defining $J$ in Definition \ref{def:cmeet} to be $\bigcap_\lambda \ker q_\lambda$. The equivalence of this definition to Definition \ref{def:cjoin} is left to the reader.

\begin{corollary}
$\text{C*-Lat($\A$)}\equiv (\text{C*-Lat($\A$)}, \preceq)$ forms a complete lattice with respect to the join and meet operations defined in Definitions \ref{def:cjoin} and \ref{def:cmeet}.
\end{corollary}


\subsection{$\al$-Admissibility and the Complete Lattice of C*-Covers}


Given a dynamical system $(\A, G, \al)$, E. Katsoulis and C. Ramsey prove in \cite{Kat-1} that $\Ce\A$ and $\Cmax\A$ are always $\al$-admissible by their respective universal properties. So determining the $\al$-admissibility of a C*-cover $(\mathcal{C},j)$ for $\A$ can be thought of as an automorphism lifting problem from the quotient $\Ce{\A}$ to $\mathcal{C}$ or an automorphism factoring problem from $\Cmax{\A}$ to $\mathcal{C}$. We approach $\al$-admissibility from the latter perspective to show that the equivalence classes of $\al$-admissible C*-covers form a complete sublattice of $\text{C*-Lat($\A$)}$.

\begin{definition}
Let $(\A, G, \al)$ be a dynamical system and let $(\mathcal{C}, j)$ be an $\al$-admissible C*-cover for $\A$. A closed two-sided ideal $J$ in $\mathcal{C}$ is called an $\al$-\textit{invariant ideal} in $\mathcal{C}$ if $J$ is invariant under each extended automorphism $\tilde{\al}_s \in \Aut(\Cc)$, i.e. $\tilde{\al}_s(J)\subseteq J$ for all $s \in G$.  A closed two-sided ideal $J$ in $\mathcal{C}$ is called an $\A$-\textit{invariant ideal} in $\mathcal{C}$ if $J$ is invariant under each $*$-automorphism of $\mathcal{C}$ that leaves $j(\A)$ invariant in $\mathcal{C}$. 
\end{definition}

\begin{example}\label{ex:K-invar}
The Toeplitz algebra $\mathcal{T}=C^*(T_z)$ together with the restriction of the symbol map $T:C(\T) \to \mathcal{T}$ defined by $T_f := P_{H^2(\T)}M_f$ is a C*-cover for the disc algebra $A(\D)=\{f\in C(\overline{\D}): f\vert_\D\text{ is analytic}\}$. Indeed, $\mathcal{T}=C^*(T(A(\D))$ and the algebra homomorphism $T\vert_{A(\D)}$ is completely isometric by the matricial version of von Neumann's inequality. 

Suppose $(A(\D), G, \al)$ is a dynamical system such that $(\mathcal{T}, T\vert_{A(\D)})$ is $\al$-admissible. Since $*$-automorphisms of the Toeplitz algebra are inner in $\B(H^2(\T))$, the compact operators $\mathbb{K}$ in $\mathcal{T}$ are invariant under the $*$-automorphisms of $\mathcal{T}$. Hence, $\mathbb{K}$ is $\al$-invariant in $\mathcal{T}$, and it follows that $\mathbb{K}$ is $A(\mathbb{D})$-invariant in $\mathcal{T}$.
\end{example}

We observe that the $\A$-invariant ideals of $\Cmax{\A}$ characterize the C*-covers for $\A$ that are $\al$-admissible for any dynamical system $(\A, G, \al)$. We call these C*-covers \textit{always admissible} for $\A$.

\begin{theorem}\label{thm:admisscorrespondence} Let $(\A, G, \al)$ be a dynamical system and let $(\mathcal{D}, i)$ be an $\al$-admissible C*-cover for $\A$. If $(\mathcal{C},j)$ is any C*-cover for $\A$ such that there exists a $*$-homomorphism $\pi$ of $\mathcal{D}$ onto $\Cc$ satisfying $\pi\circ i = j$, i.e. $(\mathcal{C},j) \preceq (\mathcal{D},i)$, then the following are equivalent:
\begin{enumerate}
\item  $(\mathcal{C}, j)$ is $\al$-admissible.
\item $\ker \pi$ is an $\al$-invariant ideal in $\mathcal{D}$. 
\end{enumerate}
\begin{proof}
Since $(\mathcal{D}, i)$ is $\al$-admissible, there exists a strongly continuous group representation $\beta:G \to \Aut(\mathcal{D})$ such that $\beta_s\circ i = i\circ \al_s $ for all $s \in G$. 

If $(\mathcal{C}, j)$ is $\al$-admissible, there exists another strongly continuous group representation $\tilde{\al}:G \to \Aut(\mathcal{C})$ such that $\tilde{\al}_s\circ j = j\circ \al_s$ for all $s \in G$. Since $\pi\circ i =j$, $\pi$ intertwines the actions $\beta$ and $\tilde{\al}$ of $G$ on $i(\A)$ and $j(\A)$, respectively. That is, for all $s \in G$, we have $(\pi\circ \beta_s)\circ i = \tilde{\al}_s \circ j.$ 

Since the image of an approximate unit for $\A$ under $i$ is an approximate unit for $\mathcal{D}$, any monomial in $i(\A)$ and $i(\A)^*$ can be written as the norm-limit of products of elements of the form $i(a)i(b)^*$. Thus, if $m$ is any monomial in $i(\A)$ and $i(\A)^*$ contained in $\ker\pi$, the intertwining of $\beta$ and $\tilde{\al}$ yields $\beta_s(m)=0$ for all $s\in G$. 

Extending linearly to polynomials in $i(\A)$ and $i(\A)^*$ contained in $\ker\pi$ then appealing to continuity, it follows that each $\beta_s$ leaves $\ker\pi$ invariant. Hence, $\ker\pi$ is an $\al$-invariant ideal.  

Conversely, if $\ker\pi$ is $\al$-invariant, then each $\beta_s$ factors through the quotient to a $*$-automorphism $\tilde{\al}_s$ that makes the below diagram commute.

\begin{center}
\begin{tikzcd}
\mathcal{D}\arrow{r}{\beta_s}\arrow[swap]{d}{\pi} & \mathcal{D}\arrow{d}{\pi}\\
\mathcal{C}\isom \mathcal{D}/\ker\pi \arrow[dashed]{r}{\tilde{\al}_s}& \mathcal{D}/\ker\pi \isom\mathcal{C}
\end{tikzcd}
\end{center}

The map $\tilde{\al}:G \to \Aut(\mathcal{C})$ given by $s \mapsto \tilde{\al}_s$ defines a strongly continuous group representation, and thus $(\mathcal{C}, G, \tilde{\al})$ is a C*-dynamical system.

An easy computation that appeals to the intertwining of $j$ and $i$ via $\pi$ and the $\al$-admissibility of $(\mathcal{D},i)$ yields that $\tilde{\al}_s\circ j = j\circ \al_s$ for all $s \in G$. Thus, $(\mathcal{C},j)$ is $\al$-admissible.
\end{proof}
\end{theorem}

\begin{example}
Suppose $(A(\D),G,\al)$ is a dynamical system such that $(\mathcal{T}, T\vert_{A(\D)})$ is $\al$-admissible. In Example \ref{ex:K-invar}, we saw that the compact operators $\mathbb{K}$ on $H^2(\T)$ is an $\al$-invariant ideal for $A(\D)$ in $\mathcal{T}$. Thus, Theorem \ref{thm:admisscorrespondence} says $\mathcal{T}/\mathbb{K} \isom C(\T)$ (with the appropriate complete isometry) is $\al$-admissible, which comes as no surprise since $(C(\T), (\cdot)\vert_\T)$ is the C*-envelope for $A(\D)$.
\end{example}

Theorem \ref{thm:admisscorrespondence} implies that every $\al$-admissible C*-cover $(\mathcal{C},j)$ for an operator algebra $\A$ must contain a nontrivial $\A$-invariant ideal when $(\Cc,j)\not\simeq \Ce{\A}$, namely the Shilov boundary ideal for $\A$ in $\Cc$. The $\A$-invariance property of the Shilov boundary ideal can be found in the literature (see Lemma 3.11 in \cite{Kat-1}, for example). However, its formulation in terms of the $\al$-admissibility/$\al$-invariance correspondence given in Theorem \ref{thm:admisscorrespondence} is new, and it yields our characterization in Corollary \ref{cor:Cmaxalwaysadmiss}.

\begin{corollary}\label{cor:Cmaxalwaysadmiss}
Let $(\A, G, \al)$ be a dynamical system and suppose $(\mathcal{C},j)$ is a C*-cover for $\A$. If $\pi:\Cmax{\A} \to \mathcal{C}$ is the quotient map of $\Cmax{\A}$ onto $\mathcal{C}$ such that $\pi \circ \imax = j$, then the following are equivalent:
\begin{enumerate}
\item  $(\mathcal{C}, j)$ is $\al$-admissible.
\item $\ker \pi$ is an $\al$-invariant ideal in $\Cmax{\A}$.
\end{enumerate}
Moreover, $(\mathcal{C}, j)$ is always admissible for $\A$ if and only if $\ker\pi$ is $\A$-invariant in $\Cmax{\A}$.
\end{corollary}

The following corollary is a restatement of Corollary \ref{cor:Cmaxalwaysadmiss} in terms of boundary ideals, explicitly. It is a useful perspective to keep in mind.%

\begin{corollary}\label{cor:invar-admiss}
Suppose $J$ is a boundary ideal for $\A$ in $\Cmax{\A}$ and let $\pi: \Cmax{\A} \to \Cmax{\A}/J$ be the natural quotient map. The following are equivalent: 
\begin{enumerate}
\item $J$ is an $\al$-invariant ideal.
\item $(\Cmax{\A}/J, \pi\circ i_{\text{max}})$ is an $\al$-admissible C*-cover for $\A$.
\end{enumerate}
\begin{proof}
Suppose $J$ is $\al$-invariant. Since $J$ is a boundary ideal for $\A$ in $\Cmax{\A}$, $(\Cmax{\A}/J, \pi\circ \imax)$ is a C*-cover for $\A$. Since $J = \ker \pi$ is an $\al$-invariant ideal in $\Cmax{\A}$, Corollary \ref{cor:Cmaxalwaysadmiss} says $(\Cmax{\A}/J, \pi\circ \imax)$ is $\al$-admissible.

The converse follows directly from Theorem \ref{thm:admisscorrespondence}.
\end{proof}
\end{corollary}

Theorem \ref{thm:admisscorrespondence} and its corollaries give a strategy for producing examples of always admissible C*-covers. We use Corollary \ref{cor:invar-admiss} to study the collection of C*-covers for $\mathcal{A}$ that are always admissible for $\A$. We will see that the collection of all such C*-covers forms a complete sublattice of $\text{C*-Lat($\A$)}$. Note that if a representative for an equivalence class in $\text{C*-Lat($\A$)}$ is $\al$-admissible, then every representative for that equivalence class is $\al$-admissible by the definition of C*-cover equivalency.

\begin{theorem}\label{thm:admiss-lattice} Let $(\A, G, \al)$ be a dynamical system. Denote the set of all equivalence classes of $\al$-admissible C*-covers for $\A$ by $\text{C*-Lat($\A,G,\al$)}.$ Then $\text{C*-Lat($\A,G,\al$)} \equiv (\text{C*-Lat($\A,G,\al$)}, \preceq)$ forms a complete sublattice of $\text{C*-Lat($\A$)}$.
\begin{proof}
E. Katsoulis and C. Ramsey showed in Lemma 3.4 of \cite{Kat-1} that $\Cmax{\A}$ and $\Ce{\A}$ are always admissible via their universal properties. So $\text{C*-Lat($\A,G,\al$)}$ is non-empty and contains a maximal and minimal element with respect to the partial order on $\text{C*-Lat($\A$)}$. It remains to show that the meet and join of a collection of $\al$-admissible C*-covers remains $\al$-admissible.

Let $\mathscr{S} = \set{(\mathcal{C}_\lambda, j_\lambda): \lambda \in \Lambda}$ be a collection of $\al$-admissible C*-covers for $\A$. For each $\lambda \in I$, there exists a strongly continuous group representation $\beta^{(\lambda)}:G \to \Aut(\mathcal{C}_\lambda)$ such that $\beta^{(\lambda)}_s\circ j_\lambda = j_\lambda\circ \al_s$ for all $s \in G$. Define $\beta^\vee:G \to \Aut(\bigoplus_\lambda \Cc_\lambda)$ by $$\beta^\vee_s = \bigoplus_\lambda \beta^{(\lambda)}_s.$$ We claim that $\beta^\vee$ can be viewed as a representation of $G$ into $\Aut(\bigvee_\lambda C_\lambda)$ with the identification $\bigvee_\lambda C_\lambda\equiv \left(C^*\left(\left(\bigoplus_\lambda j_\lambda\right)(\A)\right), \bigoplus_\lambda j_\lambda\right)$. To see this, first note that for all $s \in G$, we have 
\begin{align*}
\beta_s^\vee \circ \lp \bigoplus_\lambda j_\lambda\rp = \bigoplus_\lambda \left(\beta_s^{(\lambda)}\circ j_\lambda\right) = \bigoplus_\lambda \left( j_\lambda\circ \al_s\right) = \left(\bigoplus_\lambda j_\lambda\right)\circ \al_s.
\end{align*}
Hence, $\bigoplus_\lambda j_\lambda$ intertwines $\beta^\vee$ and $\al$, and it follows that $\beta_s^\vee\left(\left(\bigoplus_{\lambda}j_\lambda\right)(\A)\right) \subseteq \left(\bigoplus_{\lambda}j_\lambda\right)(\A)$ for all $s \in G$. 
Hence, as each $\beta_s^\vee$ is a $*$-homomorphism with $\beta_s^\vee\circ \left(\bigoplus_\lambda j_\lambda\right) = \left(\bigoplus_\lambda j_\lambda\right)\circ \al_s$, for all $s \in G$ we have $$\beta_s^\vee\lp \bigvee_{\lambda} C_\lambda\rp =C^*\left(\beta_s^\vee \circ \bigoplus_\lambda j_\lambda (\A)\right)=C^*\left(\bigoplus_{\lambda} j_\lambda\circ \al_s(\A)\right) = \bigvee_{\lambda} C_\lambda.$$  Hence, we have $\beta^\vee_s\vert_{\bigvee_\lambda \mathcal{C}_\lambda}\in \Aut\left(\bigvee_\lambda \mathcal{C}_\lambda\right)$ for all $s \in G$, which proves our claim. Therefore, $\beta^\vee:G \to \Aut(\bigvee_\lambda \Cc_\lambda)$ given by $s \mapsto \beta_s^\vee\vert_{\bigvee_\lambda \Cc_\lambda}$ is a well-defined group representation. An $\ep/3$ argument shows that $\beta^\vee$ inherits strong continuity from $\al$. Thus, $\bigvee_\lambda \mathcal{C}_\lambda$ is $\al$-admissible and is a supremum for $\mathscr{S}$.

As $\Cmax{\A}$ is $\al$-admissible, there exists a strongly continuous group representation $\tilde{\al}:G \to \Aut(\Cmax{\A})$ such that $\tilde{\al}_s\circ \imax = j\circ \al_s$ for all $s \in G$. Moreover, for each $\lambda \in \Lambda$, there is a quotient map $q_\lambda$ of $\Cmax{\A}$ onto $\mathcal{C}_\lambda$ such that $q_\lambda\circ\imax=j_\lambda$. Recall that $\bigwedge_\lambda \mathcal{C}_\lambda$ is the equivalence class of $(\Cmax{\A}/J,q_J\circ \imax)$ in $\text{C*-Lat($\A$)}$, where $J$ is the norm-closure of $\sum_\lambda \ker q_\lambda$ in $\Cmax{\A}$ and $q_J$ is the quotient map of $\Cmax\A$ onto $\Cmax\A/J$. For each $\lambda \in \Lambda$, the $\al$-admissibility of $(\mathcal{C}_\lambda, j_\lambda)$ implies that $\ker q_\lambda$ is $\tilde{\al}$-invariant in $\Cmax{\A}$ by Corollary \ref{cor:Cmaxalwaysadmiss}. Thus, $J$ is $\tilde{\al}$-invariant in $\Cmax{\A}$ since $\tilde{\al}_s$ is norm-continuous for all $s \in G$. Therefore, the characterization of $\al$-admissibility in Corollary \ref{cor:invar-admiss} yields $(\Cmax{\A}/J,q_J\circ \imax)$ is $\al$-admissible. It is left to the read to verify that every representative in $\bigwedge_\lambda \Cc_\lambda$ is $\al$-admissible and $\bigwedge_\lambda \Cc_\lambda$ must be the infimum for $\mathscr{S}$.
\end{proof}
\end{theorem}

\begin{remark} It should be noted that an argument similar to the meet case in Theorem \ref{thm:admiss-lattice} can be applied to the join case by working with the representative $(\Cmax{\A}/J,q_J)$, where $J=\bigcap_\lambda \ker q_\lambda$ in $\Cmax{\A}$, in $\bigvee_\lambda C_\lambda$. The proof is shorter than the argument provided; however, the given proof is more constructive. 
\end{remark}


\section{Inner Dynamical Systems}\label{sec:inner-dynamical}


In this section, we will assume that $\A$ is unital and restrict our view to dynamical systems where the action of the group is implemented by inner automorphisms of a C*-algebra. We define an \textit{inner dynamical system} of an operator algebra relative to a C*-cover and show that admissibility of a C*-cover is always guaranteed when the action is implemented by inner automorphisms of the original operator algebra. 


\begin{definition}\label{def:inner-in}
Let ($\A, G, \al$) be a dynamical system and let $(\Cc, j)$ be an $\al$-admissible C*-cover for $\A$. Then $(\A, G, \al)$ is \textit{inner in $(\Cc, j)$} if there exists a norm-continuous unitary group representation $U: G \to \mathcal{U}(\Cc)$ such that the extended action $\tilde{\al}:G \to \Aut(\Cc)$ is implemented by $U$, i.e. $\tilde{\al}_s = \text{ad}(U_s)$ for all $s \in G$. 
\end{definition}

In the C*-algebra context, it is standard in the literature to require the unitary group representation in Definition \ref{def:inner-in} to be strictly continuous in the multiplier algebra for $\Cc$. Since our algebras are unital, the multiplier algebra for $\Cc$ is $\Cc$ itself and strict convergence is equivalent to norm convergence. 

\begin{example}\label{ex:T2-inner} Let $T_2$ be the upper $2\times 2$ triangular matrices in $M_2$, and suppose $(T_2, \Z, \al)$ is a dynamical system. Since $(M_2, \text{incl})$ is $\al$-admissible, each $\al_n$ extends to a $*$-automorphism of $M_2$. Hence, for all $n \in \Z$, there exist unitary matrices $U_n \in M_2$ such that $\al_n = \text{ad}(U_n)$. Since each $\al_n$ must leave $T_2$ invariant, an easy computation shows that each $U_n$ must be a diagonal unitary matrix. Since $\Z$ is cyclic and discrete, it follows that $U:\Z \to \mathcal{U}(M_2)$ is a norm-continuous group representation as $\al$ is a group representation. Therefore, $(T_2, \Z, \al)$ is inner in $(M_2, \text{incl})$. 
\end{example}

When a dynamical system is inner in an $\al$-admissible C*-cover, an inner C*-dynamical system arises in the classical sense, and it is well-known that the group action of an inner C*-dynamical system looks trivial in the crossed product, in which case the full (reduced) crossed product is isomorphic to the maximal (minimal) tensor product of the C*-cover with the full (reduced) group C*-algebra of $G$. See \cite{Williams-1} for reference.

Even in non-self-adjoint dynamical systems, an action implemented by unitaries in the operator algebra looks trivial in the crossed product, as we will see in Theorem \ref{thm:inner-in-itself}. We define what it means for a dynamical system to be inner with respect to the original operator algebra.

\begin{definition}
Let ($\A, G, \al$) be a dynamical system. $(\A, G, \al)$ is \textit{inner in itself} if there exists a norm continuous unitary group representation $U: G \to \mathcal{U}(\A)$ such that ${\al}_s = \text{ad}(U_s)$ for all $s \in G$.
\end{definition}

Every inner C*-dynamical system is inner in itself. When a dynamical system $(\A, G, \al)$ is inner in itself, there are enough unitaries in the unital C*-subalgebra $\A\cap \A^*$ in $\A$ to implement the action of $G$ on $\A$, which is quite restrictive. The C*-subalgebra $\A\cap \A^*$ is called the \textit{diagonal of $\A$} and is well-defined and independent of C*-cover for $\A$ since any completely isometric representation of $\A$ is $*$-isometric when restricted to $\A\cap \A^*$. See 2.1.2 in \cite{Blecher-1}. In the following theorem, we show that $(\A, G, \al)$ being inner in itself is so restrictive that every C*-cover for $\A$ is $\al$-admissible and every relative crossed product coincides with its trivial counterpart.

\begin{theorem}\label{thm:inner-in-itself}
Suppose $(\A, G, \al)$ is inner in itself and let $(\Cc, j)$ be a C*-cover for $\A$. Then $(\Cc, j)$ is $\al$-admissible and the following operator algebras are completely isometrically isomorphic $$\A \semi_{(\Cc, j), \al} G \isom \A\semi_{(\Cc, j), \iota} G,$$ where $\iota$ is the trivial action.
\begin{proof}
Let $U:G \to \mathcal{U}(\A)$ be the norm-continuous unitary group representation such that $\al_s = \text{ad}(U_s)$ for all $s \in G$. Note that the restriction $j\vert_{\A\cap \A^*}: \A\cap \A^*\to \Cc$ is a completely isometric representation between C*-algebras. Hence, $j\vert_{\A\cap \A^*}$ must be a $*$-isometric homomorphism on $\A\cap \A^*$. Since $U(G)$ is contained in $\A\cap \A^*$, it follows that $j(U_s^*)=j(U_s)^*$ for all $s \in G$ and each $j(U_s)$ is unitary in $\Cc$.

Define $\tilde{U}:G \to \mathcal{U}(\Cc)$ by $\tilde{U}_s = j(U_s)$. Then $\tilde{U}$ is a norm continuous unitary group representation into $\Cc$, and it follows that $\beta:G\to \Aut(\Cc)$ given by $\beta_s = \text{ad}(\tilde{U}_s)$ is a strongly continuous group representation. Moreover, for all $s \in G$ and $a \in \A$, we have $$\beta_s(j(a)) = \tilde{U}_sj(a)\tilde{U}_s^* = j(U_s)j(a)j(U_s)^* = j(U_s a U_s^*) = j(\al_s(a)).$$ Thus, $(\Cc, j)$ is $\al$-admissible.

Since $(\Cc, G, \beta)$ is an inner C*-dynamical system via the unitary group representation $\tilde{U}:G \to \mathcal{U}(\Cc)$, we have $\Cc\semi_\beta G$ is isomorphic to the ``trivial'' C*-crossed product $\Cc\semi_\iota G$ via the $*$-isomorphism $\varphi:\Cc\semi_\beta G \to \Cc\semi_\iota G$ defined by $$\varphi(f)(s) = f(s)\tilde{U}_s \quad \text{for all $f \in C_c(G, \Cc)_\beta$.}$$ Thus, the restriction $\phi:=\varphi\vert_{\A\semi_{(\Cc, j), \al} G}$ is a completely isometric homomorphism. It remains to see that $\phi(\A\semi_{(\Cc, j), \al} G) = \A\semi_{(\Cc, j), \iota}G \subseteq \Cc\semi_\iota G$. 

Clearly, $\phi(\A\semi_{(\Cc, j), \al} G)$ is contained in $\A\semi_{(\Cc, j), \iota}G$ since $\tilde{U}$ is continuous and each $\tilde{U}_s$ lies in $j(\A)$. Let $f \in C_c(G, j(\A))_\iota$ be given. Since $\tilde{U}$ is a continuous representation, $g(s)=f(s)\tilde{U}_{s^{-1}} = f(s)j(U_{s^{-1}})$ is a compactly supported continuous function from $G$ to $j(\A)$, i.e. $g \in C_c(G, j(\A))_\al$. Hence, for each $s \in G$, we have $\phi(g)(s) = g(s)U_s = f(s)U_{s^{-1}}U_s = f(s)$. Therefore, $\phi$ maps a dense subset of $\A\semi_{(\Cc, j), \al} G$ onto a dense subset of $\A \semi_{(\Cc, j), \iota} G$, and thus, $\phi$ maps onto $\A\semi_{(\Cc, j), \iota} G$ by continuity.
\end{proof}
\end{theorem}

An analogous result from C*-crossed product theory follows from Theorem \ref{thm:inner-in-itself}. S. Harris and S. Kim outline the proof in \cite[Example 2.4]{Harris-Kim} in the context of trivial actions. We include the proof for convenience. 

\begin{corollary}
If $(\A, G, \al)$ is inner in itself, then $$\A \semi_\al^r G \isom \A\otimes_{\text{min}} C^*_r(G) \text{ and } \A \semi_\al G \isom \A\otimes_{\text{max}} C^*(G).$$
\begin{proof}
Theorem \ref{thm:inner-in-itself} says $\A \semi_\al^r G \isom \A\semi_{\Ce{\A},\iota}G$ and $\A \semi_\al G \isom \A\semi_{\Cmax{\A},\iota}G$, where $\iota$ is the trivial action. One can check that the maps
\begin{align*}
    \A \semi_\al^r G \ni a\lambda_s \mapsto a\otimes \lambda_s \in \A \otimes_{\text{min}} C^*_r(G)\\
    \A \semi_\al G \ni a\delta_s \mapsto a\otimes \delta_s \in \A \otimes_{\text{max}} C^*(G)
\end{align*}
extend to completely isometric isomorphisms.
\end{proof}
\end{corollary}

We give an example of a dynamical system that is not inner in itself but is inner in its C*-envelope. We show that the corresponding crossed product is non-trivial.

\begin{example}
Let $A_4$ be the four-cycle algebra in $M_4$. Theorem \ref{thm:inner-in-itself} yields $A_4 \semi_{\iota} \Z/2\Z \isom A_4 \oplus A_4$ as operator algebras. Consider the dynamical system $(A_4, \Z/2\Z, \text{ad}(u \oplus u))$, where $u=\begin{bmatrix*}[c]0 & 1\\ 1 & 0\end{bmatrix*}\in M_2$. Then $(A_4, \Z/2\Z, \text{ad}(u\oplus u))$ is not inner in itself but is inner in $\Ce{A_4} = M_4$.

Set $v = E_{11}-E_{22}\in M_2$ and define $\varphi:M_4 \semi_{\text{ad}(u\oplus u)} \Z/2\Z \to M_4 \oplus M_4$ by $$\varphi(a\delta_0 + b\delta_1) = \begin{bmatrix*}[c]a + bu & 0\\ 0 & a-bu\end{bmatrix*}= I_2\otimes a+ v\otimes bu.$$ Since $u= u^*$, it follows that $\varphi$ is a $*$-isomorphism and its restriction $\varphi\vert_{A_4 \semi_{\text{ad}(u\oplus u)} \Z/2\Z}$ is a completely isometric representation of the non-self-adjoint crossed product $A_4 \semi_{\text{ad}(u\oplus u)} \Z/2\Z$. Hence, we can identify $A_4 \semi_{\text{ad}(u\oplus u)} \Z/2\Z$ with its image $\A:=\varphi(A_4 \semi_{\text{ad}(u\oplus u)} \Z/2\Z)$ in $M_4 \oplus M_4$:  
$$\A = \set{\begin{bmatrix*}[c]A & B & & \\ & C & &\\ & & vAv & D\\ & & & vCv \end{bmatrix*} : A, B, C, D \in M_2}.$$ 

One can show that the diagonal of $\A$ is $\A \cap \A^* \isom M_2 \oplus M_2$ while the diagonal of $A_4 \oplus A_4$ is isomorphic to $\C^8$. Therefore, $A_4 \semi_{\text{ad}(u\oplus u)} \Z/2\Z$ is not isomorphic to $A_4 \oplus A_4 \isom A_4 \semi_{\iota} \Z/2\Z$.
\end{example}

Even if the action $\al:G \curvearrowright \A$ is implemented by inner automorphisms in some C*-cover $(\Cc, j)$, the existence of a unitary group representation is not guaranteed. We call dynamical systems whose action is implemented by inner automorphisms in a C*-cover \textit{locally inner}.

\begin{definition}
Let ($\A, G, \al$) be a dynamical system and suppose $(\Cc, j)$ is a C*-cover for $\A$.
\begin{enumerate}
\item $(\A, G, \al)$ is \textit{locally inner in $(\Cc, j)$} if for each $s \in G$ there exists a unitary $U_s \in \mathcal{\Cc}$ such that $\al_s = j^{-1}\circ \text{ad}(U_s)\circ j.$
\item $(\A, G, \al)$ is \textit{locally inner in itself} if for each $s \in G$ there exists a unitary $U_s \in \A\cap \A^*$ such that $\al_s = \text{ad}(U_s)$.
\end{enumerate}
\end{definition}

\begin{example}\label{ex:T2-locally-inner}
Suppose $(T_2, G, \al)$ is a dynamical system. One can show that each $\al_s=\text{ad}(U_s)$ for some diagonal unitary $U_s\in T_2$, i.e., $(T_2, G, \al)$ is locally inner in itself.
\end{example}

\begin{remark}In this context, being locally inner says that the map $s \mapsto U_s$ is a projective representation. See Appendix D.3 in \cite{Williams-1} as a reference for projective representations.
\end{remark}

Locally inner dynamical systems respect the lattice of C*-covers nicely. In particular, the existence of inner automorphisms in a C*-cover that implement the action is sufficient for that C*-cover (and any C*-cover below it) to be admissible.

\begin{theorem}\label{thm:locally-inner-admiss}
Let $(\A, G, \al)$ be a dynamical system that is locally inner in some C*-cover $(\Cc, j)$ for $\A$. Then the following hold:
\begin{enumerate}
\item $(\Cc, j)$ is $\al$-admissible.
\item If $(\mathcal{D}, i)$ is any C*-cover such that $(\mathcal{D}, i)\preceq (\Cc, j)$, then $(\mathcal{D},i)$ is $\al$-admissible. In particular, $(\A, G, \al)$ is locally inner in $(\mathcal{D},i)$.
\end{enumerate}
\begin{proof}
Since $(\A, G, \al)$ is locally inner in $(\Cc, j)$, there exist unitaries $\set{U_s}_{s\in G} \subseteq \Cc$ such that $\al_s = j^{-1}\circ \text{ad}(U_s)\circ j$ for all $s \in G$. Define $\beta:G \to \Aut(\Cc)$ by $\beta_s = \text{ad}(U_s)$. Then $\beta$ is a strongly continuous group representation since it extends $\al$ and $\Cc$ is generated by $j(\A)$. By definition, we have $\beta_s \circ j = j \circ \al_s$ for all $s \in G$ . Thus, $(\Cc, j)$ is $\al$-admissible.

Suppose $(\mathcal{D}, i)$ is a C*-cover for $\A$ such that $(\mathcal{D}, i) \preceq (\Cc, j)$. Then there exists a $*$-epimorphism $\pi$ of $\Cc$ onto $\mathcal{D}$ such that $\pi \circ j = i$. Hence, we have $\pi(U_s)$ is unitary in $\mathcal{D}$ for all $s \in G$ and the map $\gamma:G \to \Aut(\mathcal{D})$ defined by $\gamma_s = \text{ad}(\pi(U_s))$ is a strongly continuous group representation. Let $s \in G$ and $a \in \A$ be given. Since $\pi \circ j = i$, we have $$\gamma_s(i(a)) = \pi(U_s)i(a)\pi(U_s)^* =\pi(U_s)\pi(j(a))\pi(U_s^*) = \pi(U_sj(a)U_s^*).$$ As $\beta_s = \text{ad}(U_s)$ and $\beta_s \circ j = j\circ \al_s$, it immediately follows that $$\gamma_s(i(a))= \pi(\beta_s(j(a))) = \pi(j(\al_s(a)))= i(\al_s(a)).$$ Therefore, $(\mathcal{D}, i)$ is $\al$-admissible.
\end{proof}
\end{theorem}

An immediate corollary of Theorem \ref{thm:locally-inner-admiss} says admissibility of any C*-cover fails to be an obstacle when the original dynamical system is locally inner in itself.

\begin{corollary}
Let $(\A, G, \al)$ be a dynamical system that is locally inner in itself. Then every C*-cover for $\A$ is $\al$-admissible.
\begin{proof}
Let $(\Cc, j)$ be a C*-cover for $\A$. Since $(\A, G, \al)$ is inner in itself, there exist unitaries $\set{U_s}_{s\in G} \subseteq \A\cap\A^*$ such that $\al_s = \text{ad}(U_s)$. As in Theorem \ref{thm:inner-in-itself}, we note that $j$ must be a $*$-homomorphism on $\A \cap \A^*$. Hence, $\set{j(U_s)}_{s\in G}$ are unitaries in $\mathcal{C}$ such that for all $s \in G, a \in \A$ we have $$(j^{-1}\circ \text{ad}(j(U_s)) \circ j)(a) = j^{-1}(j(U_s)j(a)j(U_s)^*) = j^{-1}(j(U_saU_s^*)) =\al_s(a).$$ Therefore, $(\A, G, \al)$ is locally inner in $(\Cc, j)$, and it follows that $(\Cc, j)$ is $\al$-admissible by Theorem \ref{thm:locally-inner-admiss}.
\end{proof}
\end{corollary}

We saw that every dynamical system for $T_2$ was locally inner in itself in Example \ref{ex:T2-locally-inner}. This gives us another immediate corollary to Theorem \ref{thm:locally-inner-admiss}.

\begin{corollary}\label{cor:T2-always-admiss}
Let $(T_2, G, \al)$ be a dynamical system. If $(\Cc,j)$ is a C*-cover for $T_2$, then $(\Cc,j)$ is $\al$-admissible. In other words, every C*-cover for $T_2$ is always admissible for $T_2$.
\end{corollary}

The universal properties for the C*-envelope and maximal C*-cover yield a final corollary for this section. 

\begin{corollary}
Let $(\A, G, \al)$ be a dynamical system. Then the following hold:
\begin{enumerate}
\item If $(\A, G, \al)$ is locally inner in $\Cmax{\A}$, then $(\A, G, \al)$ is locally inner in every C*-cover for $\A$. Hence, every C*-cover for $\A$ is $\al$-admissible.
\item If $(\A, G, \al)$ is not locally inner in $\Ce{\A}$, then $(\A, G, \al)$ is not locally inner in any C*-cover for $\A$.
\end{enumerate}
\end{corollary}

\section{A New Example of a Non-Admissible C*-Cover}\label{sec:ex-non-admiss}


E. Katsoulis and C. Ramsey give an example of a C*-cover for $A(\D)$ in \cite{Kat-3} that is non-admissible for an action of $\Z$. We construct a new example of a non-admissible C*-cover in Example \ref{ex:fdnotadmiss}, which is the first finite-dimensional example in the literature.

\begin{example}\label{ex:fdnotadmiss}
Recall that the \textit{four cycle algebra} $A_4$ is the subalgebra of $M_4$ given by $$A_4=\begin{bmatrix}\ast & & \ast & \ast\\ & \ast & \ast & \ast\\ & & \ast & \\ & & & \ast \end{bmatrix}.$$ Let $\set{E_{rc}}_{r,c=1}^4$ be the standard matrix units for $M_4$ and let $I_4$ be the identity matrix in $M_4$. Consider the dynamical system $(A_4, \Z/2\Z, \text{ad}(u \oplus u))$, where $u=\begin{bmatrix*}[c]0 & 1\\ 1 & 0\end{bmatrix*}\in M_2$. Set $p = I_4 + E_{14} + E_{41} \in M_4$. Since $p$ is a positive matrix, the ``Schur product by $p$'' map $S_p: M_4 \to M_4$ is unital and completely positive. It follows that $S_p$ is completely contractive since $S_p$ is unital.

Let $a, b \in A_4$ be given. Consider 
{\footnotesize \begin{align*}
S_p(ab) &= S_p\left( \begin{bmatrix*}[c]a_{11} & 0 & a_{13} & a_{14}\\ 0 & a_{22} & a_{23} & a_{24}\\ 0 & 0 & a_{33} & 0\\ 0 & 0 & 0 & a_{44}\end{bmatrix*}\begin{bmatrix*}[c]b_{11} & 0 & b_{13} & b_{14}\\ 0 & b_{22} & b_{23} & b_{24}\\ 0 & 0 & b_{33} & 0\\ 0 & 0 & 0 & b_{44}\end{bmatrix*}\right)\\
&= S_p\left( \begin{bmatrix*}[c] a_{11}b_{11} & 0 & a_{11}b_{13} + a_{13}b_{33} & a_{11}b_{14}+a_{14}b_{44}\\ 0 & a_{22}b_{22} & a_{22}b_{23}+a_{23}b_{33} & a_{22}b_{24}+a_{24}b_{44}\\ 0 & 0 & a_{33}b_{33} & 0 \\ 0 & 0 & 0 & a_{44}b_{44} \end{bmatrix*}\right)\\
&= \begin{bmatrix*}[c] a_{11}b_{11} & 0 & 0 & a_{11}b_{14}+a_{14}b_{44}\\ 0 & a_{22}b_{22} & 0 & 0\\ 0 & 0 & a_{33}b_{33} & 0 \\ 0 & 0 & 0 & a_{44}b_{44} \end{bmatrix*}=\begin{bmatrix*}[c] a_{11} & 0 & 0 & a_{14}\\ 0 & a_{22} & 0 & 0\\ 0 & 0 & a_{33} & 0 \\ 0 & 0 & 0 & a_{44} \end{bmatrix*}\begin{bmatrix*}[c] b_{11} & 0 & 0 & b_{14}\\ 0 & b_{22} & 0 & 0\\ 0 & 0 & b_{33} & 0 \\ 0 & 0 & 0 & b_{44} \end{bmatrix*}\\
&= S_p(a)S_p(b).
\end{align*}}
Hence, $S_p\vert_{A_4}$ is a unital completely contractive (ucc) representation of $A_4$ into $M_4$.\\

Define $j:A_4 \to M_4 \oplus M_4$ by $$j(a) = \begin{bmatrix*}[c]a & 0 \\ 0 & S_p(a)\end{bmatrix*}.$$ Since $j$ is the direct sum of the ucc homomorphism $S_p$ and the unital completely isometric (ucis) inclusion $A_4 \hookrightarrow M_4$, we have that $j$ is a ucis homomorphism. It follows that $(C^*(j(A_4)), j)$ is a C*-cover for $A_4$, where $C^*(j(A_4))=M_4 \oplus C^*(S_p(A_4))\isom M_4 \oplus M_2 \oplus \C^2$ is a C*-subalgebra of $M_4 \oplus M_4$. We claim $(C^*(j(A_4)), j)$ is not $\text{ad}(u \oplus u)$-admissible. 

Towards a contradiction, assume that $(C^*(j(A_4)), j)$ is $\text{ad}(u \oplus u)$-admissible. Then there exists a $*$-automorphism $\beta: C^*(j(A_4)) \to C^*(j(A_4))$ such that $\beta\circ j = j\circ \text{ad}(u \oplus u)$. Observe that we can write $E_{14}\oplus 0 \in C^*(j(A_4))$ as
\begin{align*}
E_{14} \oplus 0 = (E_{13}\oplus 0)(E_{32}\oplus 0)(E_{24} \oplus 0) = j(E_{13})j(E_{23})^*j(E_{24}).
\end{align*}
Hence, we can write $0 \oplus E_{14}$ as $$0\oplus E_{14} =E_{14}\oplus E_{14}-E_{14}\oplus 0= j(E_{14})-j(E_{13})j(E_{23})^*j(E_{24}).$$ Since we assumed $\beta\circ j = j\circ \text{ad}(u \oplus u)$, we must have
\begin{align*}
\beta(0\oplus E_{14}) &= \beta(j(E_{14})-j(E_{13})j(E_{23})^*j(E_{24}))\\
&=\beta(j(E_{14}))-\beta(j(E_{13}))\beta(j(E_{23}))^*\beta(j(E_{24}))\\
&= j(\text{ad}(u \oplus u)E_{14})-j(\text{ad}(u \oplus u)E_{13})j(\text{ad}(u \oplus u)E_{23})^*j(\text{ad}(u \oplus u)E_{24})\\
&=j(E_{23})-j(E_{24})j(E_{14})^* j(E_{13})\\
&=(E_{23}\oplus 0)-(E_{24}\oplus 0)(E_{41}\oplus E_{41})(E_{13}\oplus 0)\\
&=E_{23}\oplus 0 - E_{23}\oplus 0\\
&=0.
\end{align*}
But $\beta$ is injective so we have a contradiction. Therefore,  $(C^*(j(A_4)), j)$ is not $\text{ad}(u \oplus u)$-admissible.
\end{example}

In Example \ref{ex:fdnotadmiss}, the algebra is finite-dimensional and being acted upon by the smallest nontrivial group. Moreover, the C*-envelope is simple so the action of the group on the given algebra is inner in its C*-envelope. Since admissibility fails for such an ``uncomplicated'' dynamical system, we really shouldn't expect admissibility of a C*-cover in general.


\section{Decomposition of Complete Isometries and Recovering Dynamics Using Partial Actions}\label{sec:partial-actions}


Admissibility failed for the C*-cover in Example \ref{ex:fdnotadmiss} since the generating complete isometry decomposed into the direct sum of a complete isometry and a proper complete contraction. We will see in this section that this decomposition \textit{always} happens when the Shilov boundary ideal is maximal and not essential. An ideal $I$ in a C*-algebra $\Cc$ is \textit{essential} when $I$ intersects each nonzero ideal of $\Cc$ nontrivially. Equivalently, $I$ is essential in $\Cc$ if the annihilator ideal for $I$ in $\Cc$ given by $I^\perp = \set{x \in \Cc : xI=0}$ is trivial.

Though the decomposition of the generating representation in this way seems to be an obstruction for admissibility (as seen in Example \ref{ex:fdnotadmiss}), we will see that there is a natural partial action that allows us to recover dynamics for a class of C*-covers, even in the case that our C*-cover is not admissible. We begin with a lemma.

\begin{lemma}\label{lem:annih-is-Ce}
Let $(\Cc,j)$ be a C*-cover for $\A$ such that the Shilov boundary ideal $\J$ for $\A$ in $\Cc$ is maximal and not essential. Then $\mathcal{J}^\perp$ is $*$-isomorphic to $\Ce{\A}$.
\begin{proof}
Since $\mathcal{J}$ is not essential, $\mathcal{J}^\perp$ is a nontrivial closed two-sided ideal in $\mathcal{C}$. As $\mathcal{J}^\perp \cap \mathcal{J}=0$, we have $\mathcal{J}^\perp$ is $*$-isomorphic to a closed two-sided ideal of $\Ce{\A}\isom \Cc/\J$.  
But $\Ce{\A}$ is simple since $\J$ is maximal so we must have $\mathcal{J}^\perp \isom \Ce{\A}$. 
\end{proof}
\end{lemma}

We give sufficient conditions for a C*-cover decomposing as the direct sum of an always admissible C*-subcover and an ideal that vanishes in the quotient.

\begin{theorem}\label{prop:decomp-simple-case}
Let $(\Cc, j)$ be a C*-cover for $\A$ such that the Shilov boundary ideal $\mathcal{J}$ for $\A$ in $\Cc$ is maximal and not essential. Then 
\begin{enumerate}
\item $\Cc \isom \Ce{\A}\oplus \J$, and
\item $j=j_1 + j_2$, where $(\Ce{\A}\oplus 0, j_1)$ is a C*-cover for $\A$ and $j_2:\A \to 0 \oplus \J$ is a proper completely contractive homomorphism.
\end{enumerate}

\begin{proof}
Since $\J$ is maximal and not essential, Lemma \ref{lem:annih-is-Ce} says $\J^\perp \isom \Ce{\A}$. Thus, there exists a unital completely isometric homomorphism $\varphi: \A \to \J^\perp$ such that $(\J^\perp, \varphi)$ is a C*-cover for $\A$. By the GNS theorem, there exists a non-degenerate faithful $*$-representation $\pi$ of $\mathcal{J}^\perp$ on some Hilbert space $\Hil$. By Theorem II.7.3.9 in \cite{Blackadar-en}, $\pi$ extends uniquely to a $*$-homomorphism $\tilde{\pi}$ from $\mathcal{C}$ to $\BH$, which yields the following diagram. Moreover, $\ker \tilde{\pi} = \lp\mathcal{J}^\perp\rp^\perp$ since $\pi$ is faithful so $\J$ is contained in $\ker\tilde{\pi}$ and $\ker\tilde{\pi}/\mathcal{J}$ is a closed two-sided ideal in $\mathcal{C}/\mathcal{J} \isom \Ce{\A}$.
\begin{center}
\begin{tikzcd}
& \mathcal{C} \arrow[dashed]{dr}{\tilde{\pi}}& \\
\A \arrow{ur}{j} \arrow{r}[swap]{\varphi} & \mathcal{J}^\perp\arrow{u}{}\arrow{r}[swap]{\pi} & \BH
\end{tikzcd}
\end{center}
As $\J$ is maximal, we must have $\ker\tilde{\pi}/\mathcal{J}$ is trivial, i.e. $\ker\tilde{\pi}=\mathcal{J}$ or $\ker\tilde{\pi}=\mathcal{C}$. Since $\tilde{\pi}\vert_{\mathcal{J}^\perp} = \pi$ is a faithful $*$-homomorphism, $\tilde{\pi}$ is completely isometric on $\varphi(\A)$ in $\J^\perp$. Hence, we must have $\lp\J^\perp\rp^\perp=\ker\tilde{\pi}=\mathcal{J}$. 

Set $p:=\varphi(1_\A)$. Then $p$ is a central projection in $\Cc$ and  $\varphi:\A \to \J^\perp$ is completely isometric. Similarly, $p^\perp:= 1_{\Cc}-p$ is a central projection in $\Cc$ and a unit for $\J$. Thus, we obtain the decomposition $$\Cc = p\Cc p \oplus p^\perp \Cc p^\perp = \J^\perp \oplus \J.$$ By identifying $\J^\perp$ with $\Ce{\A}$, we obtain the decomposition in (i) of the theorem.

Our goal for (ii) is to decompose $j:\A \to \Ce{\A}\oplus \J$ as the sum of a completely isometric homomorphism and proper completely contractive homomorphism. Define $L_p:\Cc \to \Ce{\A}\oplus 0$ by $$L_p(c) = pc.$$ Since $p$ is a central projection in $\Cc$, $L_p$ is a $*$-homomorphism of $\Cc$ onto $\Ce{\A}$ and one can show that $L_p\circ j: \A \to \Ce{\A}\oplus 0$ is completely isometric. 

Set $j_1:=L_p\circ j$ and define $j_2: \A \to 0\oplus \J$ by $j_2(a) = p^\perp j(a)$ to obtain the decomposition in (ii). 
\end{proof}
\end{theorem}

To extend dynamics in the context of Theorem \ref{prop:decomp-simple-case}, we introduce a partial action on the C*-algebra generated by the decomposed representation of $\A$. See R. Exel's textbook \cite{Exel-1} for a comprehensive resource on the theory of partial actions.

\begin{theorem}\label{thm:partial-action}
Let $G$ be a discrete group and let $(\A, G, \al)$ be a dynamical system. Suppose $(\mathcal{C},j)$ is a  C*-cover for $\mathcal{A}$ such that the Shilov boundary ideal $\J$ for $\mathcal{A}$ in $\mathcal{C}$ is maximal and not essential. If $G$ is amenable or $\A$ is Dirichlet, then there exists a partial action $\theta: G \curvearrowright \mathcal{C}$ and a norm closed subalgebra $\B \subseteq \mathcal{C}\semi_\theta G$ such that $\B$ is completely isometrically isomorphic to $\A\semi_\al G.$
\begin{proof}
By Theorem \ref{prop:decomp-simple-case}, we have $\Cc = \Ce{\A}\oplus \J$. Moreover, there exists a central projection $p \in \Cc$ such that $j=j_1 + j_2$, where $j_1:\A \to \Ce{\A}\oplus 0$ is a completely isometric homomorphism given by $j_1(a) = pj(a)$ for all $a \in \A$. 

For each $s \in G$, define $$D_s = \begin{cases} \Cc & \text{if $s = e$}\\ \Ce{\A}\oplus 0 & \text{if $s \ne e$}\end{cases},$$ where $e$ is the identity element in $G$. As $(\Ce{\A}\oplus 0, j_1)\simeq(\Ce{\A},\imin)$ is an $\al$-admissible C*-cover for $\A$, there exists an action $\beta:G \to \Aut(\Ce{\A}\oplus 0)$ such that $\beta_s\circ j_1 = j_1\circ \al_s$ for all $s \in G$. Hence, for each $s \in G$, we can define $\theta_s: D_{s^{-1}}\to D_s$ by $$\theta_s = \begin{cases} \text{id}_{\Cc} & \text{if $s = e$}\\ \beta_s & \text{if $s \ne e$}\end{cases}.$$ Thus, $\theta = \set{\set{\theta_s}_{s\in G}, \set{D_s}_{s\in G}}$ is a partial action of $G$ on $\Cc$.

Our construction yields two crossed products of C*-algebras, namely the global C*-crossed product $(\Ce{\A}\oplus 0)\semi_\beta G$ and the partial C*-crossed product $\Cc\semi_\theta G$. If $G$ is amenable or $\A$ is Dirichlet, Theorem 3.14 or 5.5 in \cite{Kat-1} yields all relative crossed products of $\A$ by $G$ coincide via a complete isometry that maps generators to generators. In particular, this implies that the full crossed product $\A \semi_\al G$ is completely isometrically isomorphic to the closure of $C_c(G, j(\A))$ in $(\Ce{\A}\oplus 0)\semi_\beta G$. Moving forward, we will identify $\A\semi_\al G$ with this subalgebra of $(\Ce{\A}\oplus 0)\semi_\beta G$. 

Recall that there exists a canonical unitary representation $U:G \to (\Ce{\A}\oplus 0)\semi_\beta G$ and observe that the map $\pi:\Cc \to (\Ce{\A}\oplus 0)\semi_\beta G$ given by $\pi(x)=(px)U_e$ is a $*$-homomorphism. To see that $(\pi, U)$ is a covariant representation, fix $e \ne s \in G$ and let $j_1(a) \in D_{s^{-1}}=\Ce{\A}\oplus 0$ be given. Since $p$ is the unit for $\Ce{\A}\oplus 0$, we have $px=x$ for all $x \in \Ce{\A}\oplus 0$, and it follows that
\begin{align*}
U_s\pi(j_1(a))U_s^* = U_s((pj_1(a))U_e)U_s^* = \beta_s(pj_1(a))U_e = p\beta_s(j_1(a))U_e = \pi(\theta_s(j_1(a))).
\end{align*}
Since $j_1(\A)$ generates $D_{s^{-1}}=\Ce{\A}\oplus 0$ as a C*-algebra, we get $U_s\pi(x)U_s^* = \pi(\theta_s(x))$ for all $x \in D_{s^{-1}}$. It is clear that $U_e\pi(x)U_{e^{-1}}=\pi(\theta_e(x))$ by definition of $\theta$. Hence, $(\pi, U)$ is a covariant representation of $\theta$ in $(\Ce{\A}\oplus 0)\semi_\beta G$. Thus, there exists an integrated form $\pi \semi U: \Cc\semi_\theta G \to (\Ce{\A}\oplus 0)\semi_\beta G$ such that $(\pi\semi U)(x\delta_s) = \pi(x)U_s$ for all $x \in \Cc, s \in G$.

Define $\varphi: j_1(\A) \to \Cc\semi_\theta G$ by $\varphi(j_1(a)) = j_1(a)\delta_e$, and define $V:G \to \Cc\semi_\theta G$ by $V_s = 1_s\delta_s$, where $1_s = p$ when $s\ne e$ and $1_e = 1_\Cc$. Fix $e \ne s \in G$. Recall that the $\al$-admissibility of $(\Ce{\A}\oplus 0, j_1)$ yields $\beta_s\circ j_1 = j_1 \circ \al_s$. Thus, for each $a \in \A$, we have
\begin{align*}
V_s\varphi(j_1(a))V_{s^{-1}} &= (p\delta_s)(j_1(a)\delta_e)(p\delta_{s^{-1}})
=(\beta_s(\beta_{s^{-1}}(p)j_1(a))\delta_s)(p\delta_{s^{-1}})\\
&=\beta_s(\beta_{s^{-1}}(p)j_1(a)p)\delta_e
=\beta_s(pj_1(a)p)\delta_e
=\beta_s(j_1(a))\delta_e\\
&=j_1(\al_s(a))\delta_e = \varphi(j_1(\al_s(a))) = (\varphi\circ j_1)(\al_s(a)).
\end{align*}

A similar computation shows $V_e\varphi(j_1(a))V_e=(\varphi\circ j_1)(\al_e(a))$ for all $a \in \A$. Since $\varphi \circ j_1$ is also completely contractive, $(\varphi\circ j_1, V)$ is a non-degenerate covariant representation of $(\A, G, \al)$ in $\Cc\semi_\theta G$. Thus, Proposition 3.7 in \cite{Kat-1} yields a completely contractive integrated form $(\varphi\circ j_1) \semi V$ of the full crossed product $\A\semi_\al G$ to $\Cc\semi_\theta G$ given by $$((\varphi\circ j_1)\semi V)(j_1(a)U_s) = \varphi(j_1(a))V_s=j_1(a)\delta_s.$$ 

Set $\B = \overline{\bigoplus_{s\in G} D_s \cap j_1(\A)} \subseteq \Cc\semi_\theta G$ and note that $((\varphi\circ j_1)\semi V)(\A\semi_\al G) \subseteq \B$ and $(\pi\semi U)(\B)\subseteq \A\semi_\al G$. We claim that $(\varphi\circ j_1) \semi V$ and $\pi \semi U\vert_{\B}$ are completely contractive inverses.

Let $B=\sum_{s\in G} j_1(a_s)\delta_s \in \B$ be given. Observe that 
\begin{align*}
((\varphi\circ j_1)\semi V)\lp (\pi \semi U)\lp B \rp\rp &= ((\varphi\circ j_1)\semi V)\lp \sum_{s\in G} \pi(j_1(a_s))U_s\rp\\
&= ((\varphi\circ j_1)\semi V)\lp \sum_{s\in G} pj_1(a_s)U_s\rp\\
&= ((\varphi\circ j_1)\semi V)\lp \sum_{s\in G} j_1(a_s)U_s\rp\\
&= \sum_{s\in G} \varphi(j_1(a_s))V_s\\
&=\sum_{s\in G}j_1(a_s)\delta_s\\
&=B.
\end{align*}

Similarly, if $A=\sum_{s\in G} j_1(a_s)U_s \in \A\semi_{\al} G$, then we have
\begin{align*}
(\pi \semi U)\lp((\varphi\circ j_1)\semi V)(A)\rp &= (\pi \semi U)\lp \sum_{s\in G} j_1(a_s)\delta_s\rp \\
&=\sum_{s\in G}\pi(j_1(a_s))U_s\\
&=\sum_{s\in G}pj_1(a_s)U_s\\
&=\sum_{s\in G}j_1(a_s)U_s\\
&=A.
\end{align*}

Thus, the integrated forms $\varphi \semi V$ and $\pi \semi U\vert_{\B}$ are completely contractive inverses on dense subsets of $\A \semi_\al G$ and $\B$. Hence, by continuity, $\varphi \semi V$ and $\pi \semi U\vert_{\B}$ are completely contractive inverses, and it follows that $\A\semi_\al G \isom \B$.
\end{proof}
\end{theorem}

Though the assumptions on the Shilov boundary ideal in Theorem \ref{thm:partial-action} are strong, the theorem can still be applied to many important examples in the literature. Indeed, Theorem \ref{thm:partial-action} applies to any operator algebra whose C*-envelope is simple. In particular, we can revisit Example \ref{ex:fdnotadmiss}.

\begin{example}
Let $(A_4, \Z/2\Z, \text{ad}(u \oplus u))$ be the dynamical system from Example \ref{ex:fdnotadmiss}. We showed that $(M_4 \oplus (M_2 \oplus \C), j)$ was not an ad$(u\oplus u)$-admissible C*-cover for the four cycle algebra $A_4$, where $j: A_4 \to M_4 \oplus (M_2 \oplus \C)$ was given by the ucis representation $j(a)=a \oplus S_p(a)$ for all $a \in A_4$ and $S_p:A_4 \to M_2 \oplus \C$ is a completely contractive homomorphism given by Schur multiplication by some positive matrix $p \in M_4$.

The Shilov boundary ideal for $A_4$ in $C^*(j(A_4)) = M_4 \oplus (M_2 \oplus \C^2)$ is $\J = 0\oplus (M_2 \oplus \C^2)$, which is maximal and not essential since $M_4 \oplus 0$ annihilates $\J$. Thus, Theorem \ref{thm:partial-action} applies and says there exists a partial action $\theta:\Z/2\Z \curvearrowright M_4 \oplus (M_2 \oplus \C)$ and a subalgebra $\B$ of the partial C*-crossed product $(M_4 \oplus (M_2 \oplus \C))\semi_\theta \Z/2\Z$ that is completely isometrically isomorphic to $\A_4 \semi_{\text{ad}(u\oplus u)} \Z/2\Z$. Working through the construction of Theorem \ref{thm:partial-action}, we can see that $$(M_4 \oplus (M_2 \oplus \C))\semi_\theta \Z/2\Z \isom (M_4 \semi_{\text{ad}(u\oplus u)} \Z/2\Z)\oplus (M_2 \oplus \C) \isom (M_4 \oplus M_4)\oplus (M_2 \oplus \C).$$ Moreover, we have $\B = (A_4 \semi_{\text{ad}(u\oplus u)} \Z/2\Z) \oplus 0$. Therefore, despite the C*-cover being non-admissible, we can extend the original dynamics to the C*-cover so long as we consider partial actions.
\end{example}

Example \ref{ex:fdnotadmiss} and the example presented in Proposition 2.1 of \cite{Kat-3} are the only examples of non-admissible C*-covers in the literature. The partial action construction in Theorem \ref{thm:partial-action} applies to both of these examples. 

\begin{example}
Consider the M\"{o}bius transformation $\tau(w)= \displaystyle\frac{w-\frac{1}{2}}{1-\frac{w}{2}}$, which is a conformal mapping of $\D$ onto $\D$, $\T$ onto $\T$, and preserves orientation. Thus, $\al_1: f\mapsto f\circ \tau$ defines a completely isometric automorphism of $A(\D)$ and induces an action of $\Z$ on $A(\D)$ that yields the dynamical system $(A(\D), \Z, \al)$.

In Proposition 2.1 of \cite{Kat-3}, E. Katsoulis and C. Ramsey show that $(C(\T)\oplus M_2, i)$, where $i: z \mapsto z\oplus \begin{bmatrix*}[r]0 & 0\\ 1 & 0\end{bmatrix*}$, is a C*-cover for $A(\D)$ that is not $\al$-admissible. However, observe that the Shilov boundary ideal for $A(\D)$ in $C(\T)\oplus M_2$ is $\mathcal{J} = 0 \oplus M_2$, which is neither maximal nor essential since $C(\T)\oplus 0 \isom \Ce{A(\D)}$ annihilates $\J$. Using the partial action constructed in Theorem \ref{thm:partial-action}, we can compute the partial C*-crossed product $$\lp C(\T)\oplus M_2\rp \semi_\theta \Z \isom \lp C(\T)\semi_\al \Z\rp \oplus M_2$$ and see that $\B = (\A(\D) \semi_\al \Z)\oplus 0$, which is completely isometrically isomorphic to $A(\D)\semi_\al \Z.$
\end{example}

Given that the Shilov boundary ideal is not essential in every known example of a non-admissible C*-cover, we might ask the following question: does the Shilov boundary ideal being essential in a C*-cover determine if that C*-cover is always admissible? We show that the Shilov boundary ideal being essential is not necessary for the C*-cover to be always admissible.

\begin{proposition}\label{prop:not-essential-always-admiss}
There exist C*-covers for $T_2$ where the Shilov boundary ideal is not essential.
\begin{proof}
Define $j: T_2 \to M_2 \oplus M_2$ by $$j\lp\begin{bmatrix*}[c]a & b\\ 0 & c\end{bmatrix*}\rp= \begin{bmatrix*}[c]a & b\\ 0 & c\end{bmatrix*}\oplus \begin{bmatrix*}[c]a & 0\\ 0 & c\end{bmatrix*}.$$ It is easy to see that $j$ is a completely isometric representation of $T_2$ and $C^*(j(T_2)) \isom M_2 \oplus \C^2$ by a matrix unit argument. Hence, $( M_2 \oplus \C^2, j)$ is a C*-cover for $T_2$. It follows that $( M_2 \oplus \C^2, j)$ is always admissible for $T_2$ by Corollary \ref{cor:T2-always-admiss}, yet the Shilov boundary ideal $\J = 0 \oplus \C^2$ in $M_2 \oplus \C^2$ is not essential since it has trivial intersection with the ideal $M_2 \oplus 0$.
\end{proof}
\end{proposition}


\bibliographystyle{amsalpha}
\bibliography{ref}
\end{document}